\documentclass[11pt]{article}
\usepackage{amsthm,geometry,amssymb,amsmath,enumerate,float,tikz,cite,algorithm2e,setspace}
\pdfoutput=1
 
\newtheorem{theorem}{Theorem}[section]
\newtheorem{lemma}{Lemma}[section]

\usepackage{lineno}

\title{Relating broadcast independence and independence }
\author{S. Bessy$^1$ \and  D. Rautenbach$^2$}
\date{}

\begin{document}
\onehalfspace

\maketitle
\vspace{-10mm}
\begin{center}
{\small $^1$ Laboratoire d'Informatique, de Robotique et de
  Micro\'{e}lectronique de Montpellier,\\ Montpellier, France,
  \texttt{stephane.bessy@lirmm.fr}\\[3mm] $^2$ Institute of
  Optimization and Operations Research, Ulm University,\\ Ulm,
  Germany, \texttt{dieter.rautenbach@uni-ulm.de}}
\end{center}

\begin{abstract}
An independent broadcast on a connected graph $G$ is a function
$f:V(G)\to \mathbb{N}_0$ such that, for every vertex $x$ of $G$, the
value $f(x)$ is at most the eccentricity of $x$ in $G$, and $f(x)>0$
implies that $f(y)=0$ for every vertex $y$ of $G$ within distance at
most $f(x)$ from $x$.  The broadcast independence number $\alpha_b(G)$
of $G$ is the largest weight $\sum\limits_{x\in V(G)}f(x)$ of an
independent broadcast $f$ on $G$.  Clearly, $\alpha_b(G)$ is at least
the independence number $\alpha(G)$ for every connected graph $G$.  Our main result
implies $\alpha_b(G)\leq 4\alpha(G)$.  We prove a tight inequality and
characterize all extremal graphs.
\end{abstract}
{\small 
\begin{tabular}{lp{13cm}}
{\bf Keywords}: & broadcast independence; independence
\end{tabular}
}

\pagebreak

\section{Introduction}

In his PhD thesis \cite{er} Erwin introduced the notions of broadcast
domination and broadcast independence in graphs, cf. also
\cite{duerhahehe}.  While broadcast domination was studied in detail
\cite{brmyte,helo,hemy,lumy,se,soko}, only little research exists on
broadcast independence \cite{ahboso,boze}.  In the present paper we
relate broadcast independence to ordinary independence in graphs; one
of the most fundamental and well studied notions in graph theory.

We consider finite, simple, and undirected graphs, and use standard
terminology and notation.  Let $\mathbb{N}_0$ be the set of
nonnegative integers.  For a connected graph $G$, a function
$f:V(G)\to \mathbb{N}_0$ is an {\it independent broadcast on $G$} if
\begin{quote}
\begin{enumerate}[(B1)]
\item $f(x)\leq {\rm ecc}_G(x)$ for every vertex $x$ of $G$, where
  ${\rm ecc}_G(x)$ is the eccentricity of $x$ in $G$, and
\item ${\rm dist}_G(x,y)>\max\{ f(x),f(y)\}$ for every two distinct
  vertices $x$ and $y$ of $G$ with $f(x),f(y)>0$, where ${\rm
    dist}_G(x,y)$ is the distance of $x$ and $y$ in $G$.
\end{enumerate}
\end{quote}
The {\it weight} of $f$ is $\sum\limits_{x\in V(G)}f(x)$.  The {\it
  broadcast independence number} $\alpha_b(G)$ of $G$ is the maximum
weight of an independent broadcast on $G$, and an independent
broadcast on $G$ of weight $\alpha_b(G)$ is {\it
  optimal}.\footnote{Note that, for a disconnected graph $G$, (B1) and
  (B2) allow to assign an arbitrarily large value to one vertex in
  each component of $G$, which means that the weight of independent
  broadcasts on $G$ would be unbounded.  To avoid this issue, ${\rm
    ecc}_G(x)$ in (B1) could be replaced by the eccentricity of $x$ in
  the connected component of $G$ that contains $x$.}  Let $\alpha(G)$
be the usual independence number of $G$, that is, $\alpha(G)$ is the
maximum cardinality of an independent set in $G$, which is a set of
pairwise nonadjacent vertices of $G$.  For an integer $k$, let $[k]$
be the set of all positive integers at most $k$, and let $[k]_0=\{
0\}\cup [k]$.

Clearly, assigning the value $1$ to every vertex in an independent set
in some connected graph $G$, and $0$ to all remaining vertices of $G$, yields an
independent broadcast on $G$, which implies
$$\alpha_b(G)\geq \alpha(G)\mbox{ for every connected graph $G$}.$$
A consequence of our main result is that
$$\alpha_b(G)\leq 4\alpha(G)\mbox{ for every connected graph $G$}.$$ The fact
that the broadcast independence number and the independence number are
within a constant factor from each other immediately implies the
computational hardness of the broadcast independence number, and also
yields efficient constant factor approximation algorithms for the
broadcast independence number on every class of graphs for which the
independence number can efficiently be approximated within a constant
factor.

In order to phrase our main result, we introduce some special graphs.
For a positive integer $k$, a graph $H$ is a {\it $k$-strip with
  partition $(B_0,\ldots,B_k)$} if $V(H)$ can be partitioned into $k$
nonempty cliques $B_0,\ldots,B_k$ such that
\begin{itemize}
\item $B_0$ contains a unique vertex $x$,
\item all vertices in $B_i$ have distance $i$ in $H$ from $x$, and
\item $B_i$ is completely joined to $B_{i+1}$ 
for every even index $i$ in $[k-1]_0$.
\end{itemize}
For a positive integer $k$, let ${\cal G}_2(k)$ be the class of all
connected graphs that arise from the disjoint union of two
$(2k+1)$-strips $H_1$ with partition $(B_0^1,\ldots,B_{2k+1}^1)$ and
$H_2$ with partition $(B_0^2,\ldots,B_{2k+1}^2)$ by adding some edges
between $B_{2k+1}^1$ and $B_{2k+1}^2$. 
An example of such a graph is depicted in Figure~\ref{fig:GraphG2}.
\begin{figure}[H]
\centering
\includegraphics[width=0.9\textwidth]{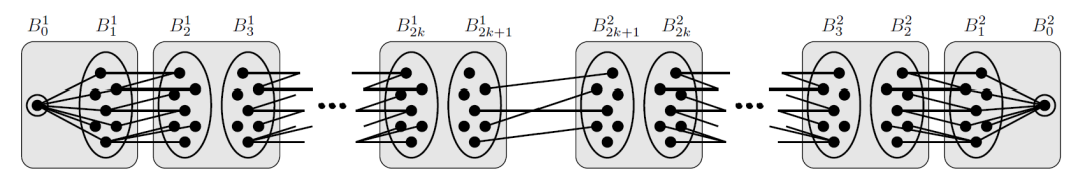}
\caption{A graph from the family ${\cal G}_2(k)$. 
The vertices in each gray box form a clique.}
\label{fig:GraphG2}
\end{figure}
For positive integers $k$ and $\ell$ with $\ell\geq 2$, let ${\cal
  G}_0(k,\ell)$ be the class of all graphs that arise from the
disjoint union of $\ell$ $2k$-strips 
$H_1,\ldots,H_{\ell}$,
where $H_i$ has partition $(B_0^i,\ldots,B_{2k}^i)$ for $i$ in $[\ell]$,
and a possibly empty set $R$ of vertices
by adding all possible edges within $R\cup
\bigcup\limits_{i=1}^{\ell}B_{2k}^i$. 
A graph from the family ${\cal G}_0(k,\ell)$ 
is depicted in Figure~\ref{fig:GraphG0}.

\begin{figure}[H]
\centering
\includegraphics[width=0.8\textwidth]{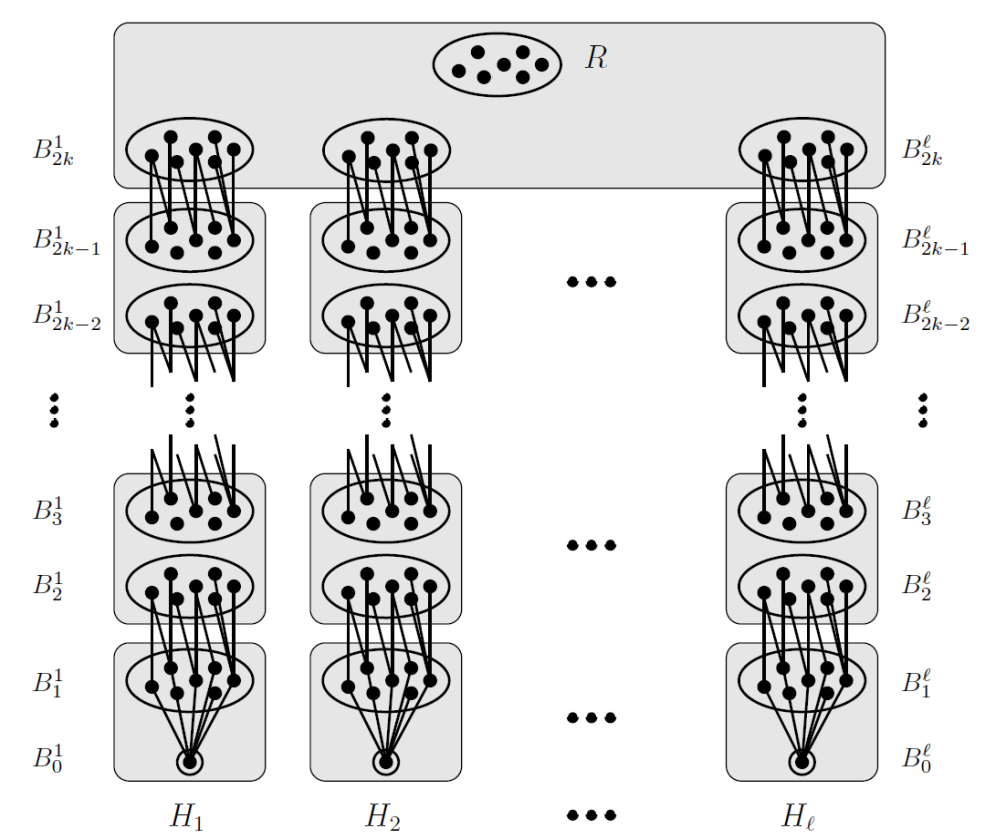}
\caption{A graph from the family ${\cal G}_0(k,\ell)$. 
Also here, the vertices in each gray box form a clique.}
\label{fig:GraphG0}
\end{figure}
Finally, let
$${\cal G}_2=\bigcup\limits_{k\geq 1}{\cal G}_2(k)
\,\,\,\,\,\,\,\,\,\,\mbox{ and }\,\,\,\,\,\,\,\,\,\, {\cal
  G}_0=\bigcup\limits_{k\geq 1}\bigcup\limits_{\ell\geq 2}{\cal
  G}_0(k,\ell).$$
The following is our main result; proofs are given
in the following section.

\begin{theorem}\label{theorem1}
If $G$ is a connected graph such that $G$ has diameter at least $3$ or
$\alpha(G)\geq 3$, and $f$ is an optimal broadcast on $G$, then
\begin{eqnarray}
\alpha_b(G)&\leq & 4\alpha(G)-4\min\left\{1,
\frac{2\alpha(G)}{f_{\max}+2}\right\},\label{e2}
\end{eqnarray}
where $f_{\max}=\max\{ f(x):x\in V(G)\}$.  Equality holds in
(\ref{e2}) if and only if $G\in {\cal G}_0\cup {\cal G}_2$.
\end{theorem}
The assumption that $G$ has diameter at least $3$ 
or $\alpha(G)\geq 3$
excludes some trivial cases; 
suppose that a nonempty connected graph $G$ has
diameter at most $2$ and $\alpha(G)\leq 2$.  If $\alpha(G)=1$, then
$G$ is a clique, which implies $\alpha_b(G)=\alpha(G)$, and, if
$\alpha(G)=2$, then (B1) and (B2) imply $\alpha_b(G)=2$, that is, both
parameters are equal in these cases.

\section{Proofs}

For the proof of Theorem \ref{theorem1}, we need some properties of
the graphs in ${\cal G}_0\cup {\cal G}_2$.

\begin{lemma}\label{lemma1}
Let $k$ and $\ell$ be positive integers with $\ell\geq 2$.
\begin{enumerate}[(i)]
\item If $G\in {\cal G}_2(k)$, then 
$\alpha(G)=2k+2$, 
$\alpha_b(G)=8k+4$, and
$\max\{ f(x):x\in V(G)\}=4k+2$
for every optimal independent broadcast $f$ on $G$.
\item If $G\in {\cal G}_0(k,\ell)$, then 
$\alpha(G)=k\ell+1$,
$\alpha_b(G)=4k\ell$, and
$\max\{ f(x):x\in V(G)\}=4k$
for every optimal independent broadcast $f$ on $G$.
\end{enumerate}
\end{lemma}
\begin{proof} 
We only give details for the proof of (ii);
the simpler proof of (i) can be obtained in a similar way.
Let $G\in {\cal G}_0(k,\ell)$. 
Let $H_1,\ldots,H_{\ell}$ be as in the definition of 
${\cal G}_0(k,\ell)$.

Since $B_{2i}^j\cup B_{2i+1}^j$ is a clique 
for every $i$ in $[k-1]_0$ and every $j\in [\ell]$, 
and since
$R\cup \bigcup\limits_{i=1}^{\ell}B_{2k}^i$ is a clique,
we obtain 
$\alpha(G)\leq k\ell+1$.
Since a set containing
one vertex from $B_{2i}^j$ 
for every $i$ in $[k-1]_0$ and every $j\in [\ell]$,
and
one vertex from $B_{2k}^1$ 
is independent,
we obtain $\alpha(G)=k\ell+1$.

Let $f$ be an optimal independent broadcast on $G$.
Let $j$ be an arbitrary index in $[\ell]$.
Let $i_1,\ldots,i_r$ be all indices such that 
$0\leq i_1<\ldots<i_p\leq 2k-1$, and 
$f$ has a positive value on some vertex $x_q$ in $B_{i_q}^j$
for every $q$ in $[p]$.
Since each $B_i^j$ is a clique, 
the vertices $x_1,\ldots,x_p$ are unique.
By the structure of $G$,
the distance between a vertex in $B_r^j$ 
and a vertex in $B_s^j$
for $r$ and $s$ with $r,s\in [2k]_0$ and $r<s$
is at most $s-r+1$, 
and at most $s-r$ if $r=0$.
Therefore, (B2) implies that 
$i_{q+1}\geq i_q+f(x_q)$ for every $q$ in $[p-1]$,
and that 
$i_2\geq i_1+f(x_1)+1$ if $p\geq 2$ and $i_1=0$.
If $p\geq 2$ and $i_1>0$, then
$$\sum\limits_{q=1}^{p-1}f(x_q)
\leq
\sum\limits_{q=1}^{p-1}(i_{q+1}-i_q)
=i_p-i_1\leq i_p-1,$$
and, if $p\geq 2$ and $i_1=0$, then 
$$\sum\limits_{q=1}^{p-1}f(x_q)
\leq
(i_2-i_1-1)+\sum\limits_{q=2}^{p-1}(i_{q+1}-i_q)
=i_p-i_1-1\leq i_p-1,$$
that is, the same bound holds in both cases.

First, we assume that 
$f$ has a positive value on some vertex $x$
in $R\cup \bigcup\limits_{i=1}^{\ell}B_{2k}^i$.
By the structure of $G$, we have $f(x)\leq {\rm ecc}_G(x)\leq 2k+1$. 
(B2) implies 
$f(x_p)\leq {\rm dist}_G(x_p,x)-1\leq 2k-i_p$.
Hence,
$\sum\limits_{q=1}^{p}f(x_q)\leq 2k-1$ if $p\geq 2$,
and 
$\sum\limits_{q=1}^{p}f(x_q)\leq 2k$ if $p=1$ and $i_1=0$.
Since $j$ was chosen arbitrarily, we obtain
$\alpha_b(G)\leq 2k\ell+2k+1.$

Next, we assume that 
$f$ is $0$ on $R\cup \bigcup\limits_{i=1}^{\ell}B_{2k}^i$.
This implies
$f(x_p)\leq {\rm ecc}_G(x_p)=4k-i_p+1\leq 4k+1$.
If $f(x_p)={\rm ecc}_G(x_p)$,
then, by (B2), 
$x_p$ is the only vertex of $G$ with a positive value of $f$,
and, hence, 
$\alpha_b(G)\leq 4k+1.$
If $f(x_p)\leq 4k-i_p$, then 
$\sum\limits_{q=1}^{p}f(x_q)\leq 4k-1$ if $p\geq 2$,
and 
$\sum\limits_{q=1}^{p}f(x_q)\leq 4k$ if $p=1$ and $i_1=0$.
Since $j$ was chosen arbitrarily, we obtain
$\alpha_b(G)\leq 4k\ell$.
Altogether, we obtain
$$\alpha_b(G)\leq \max\{ 2k\ell+2k+1,4k+1,4k\ell\}=4k\ell.$$
Since the function $f^*$
that has value $4k$ on every vertex in
$\bigcup\limits_{i=1}^{\ell}B_0^i$
and value $0$ everywhere else
is an independent broadcast on $G$
of weight $4k\ell$,
we conclude
$$\alpha_b(G)=4k\ell.$$
Since $\max\{ 2k\ell+2k+1,4k+1\}<4k\ell$,
the above arguments actually imply that $f^*$
is the unique optimal broadcast on $G$,
which completes the proof.
\end{proof}
We are now in a position to prove our main result.

\begin{proof}[Proof of Theorem \ref{theorem1}]
Let $X=\{ x\in V(G):f(x)>0\}$.
For every vertex $x$ in $X$ and every nonnegative integer $i$, 
let 
\begin{eqnarray*}
B_i(x) & = & \big\{ y\in V(G):{\rm dist}_G(x,y)=i\big\},\\
B(x)&=&\bigcup\limits_{i=0}^{\left\lfloor\frac{f(x)}{2}\right\rfloor}B_i(x),\\
\partial B(x)&=&B_{\left\lfloor\frac{f(x)}{2}\right\rfloor}(x)\mbox{, and }\\
R &=& V(G)\setminus \bigcup\limits_{x\in X}B(x).
\end{eqnarray*}
If there are two distinct vertices $x$ and $x'$ in $X$ 
such that the sets $B(x)$ and $B(x')$ intersect,
then 
$${\rm dist}_G(x,x')\leq \frac{f(x)}{2}+\frac{f(x')}{2}
\leq \max\{ f(x),f(x')\},$$
which contradicts (B2). 
Hence, 
$$\mbox{the sets $B(x)$ for $x$ in $X$ are disjoint.}$$
Note that no vertex $y$ in $B(x)\setminus \partial B(x)$
has a neighbor outside of $B(x)$.
For every $x$ in $X$, let $p(x)$ be an arbitrary vertex in $\partial B(x)$,
and let $P(x)$ be a shortest path in $G$ between $x$ and $p(x)$.
Note that $P(x)$ has order $\left\lfloor\frac{f(x)+2}{2}\right\rfloor$,
that $x$ and $p(x)$ coincide if and only if $f(x)=1$,
and that $p(x)$ is the only vertex on $P(x)$ 
that may have neighbors outside of $B(x)$.

For $i\in \{ 0,1,2,3\}$, let $X_i=\{ x\in X:f(x)\mod 4\equiv i\}$.
For every $x$ in $X_0\cup X_1$, 
the path $P(x)$ contains a unique independent set $I(x)$ 
of order $\left\lfloor\frac{f(x)+4}{4}\right\rfloor$
that contains $p(x)$,
and for every $x$ in $X_2\cup X_3$, 
the path $P(x)$ contains a unique independent set $I(x)$ 
of order $\left\lfloor\frac{f(x)+2}{4}\right\rfloor$
that does not contain $p(x)$.
The next table summarizes the different cases.

\renewcommand{\arraystretch}{1.5}
\begin{table}[H]
  \centering
  \begin{tabular}{c|c|c|c|l}
    $f(x)\mod 4$ & $\left\lfloor\frac{f(x)}{2}\right\rfloor \mod 2$ &
    $|P(x)|\mod 2$ & $|P(x)|$ & $|I(x)|$\\
    \hline
    \hline
    0 & 0 & 1 & $\frac{f(x)+2}{2}$ & $\frac{f(x)+4}{4}$
    (and $I(x)$ contains $p(x)$)\\
    \hline
    1 & 0 & 1 & $\frac{f(x)+1}{2}$ & $\frac{f(x)+3}{4}$
    (and $I(x)$ contains $p(x)$)\\
    \hline
    2 & 1 & 0 & $\frac{f(x)+2}{2}$ & $\frac{f(x)+2}{4}$
    (and $I(x)$ does not contain $p(x)$)\\
    \hline
    3 & 1 & 0 & $\frac{f(x)+1}{2}$ & $\frac{f(x)+1}{4}$
    (and $I(x)$ does not contain $p(x)$)\\
  \end{tabular}
  \caption{Values of different parameters according to $f(x)\mod 4$.}
  \label{tab:values}
\end{table}
We consider three cases.

\bigskip

\noindent {\bf Case 1} $X_0=X_3=\emptyset$.

\bigskip

\noindent Let $I=\bigcup\limits_{x\in X}I(x)$.
Suppose, for a contradiction, that $I$ is not independent.
Since $I(x)$ contains $p(x)$ only if $x$ belongs to $X_1$,
it follows that there are two distinct vertices $x$ and $x'$ in $X_1$
such that $p(x)$ is adjacent to $p(x')$.
Now, 
$${\rm dist}_G(x,x')\leq |P(x)|+|P(x')|-1 \le \frac{f(x)+1}{2}+\frac{f(x')+1}{2}-1
\leq \max\{ f(x),f(x')\},$$
which contradicts (B2). 
Hence, $I$ is independent.
Since $X=X_1\cup X_2$ using Table~\ref{tab:values} we obtain
$$|I(x)| \geq \frac{f(x)+2}{4}$$ for every $x$ in $X$.  Since
$f_{\max}\cdot |X|\geq \alpha_b(G)$, we obtain
\begin{eqnarray}\label{e7}
\alpha(G)\geq |I|=\sum\limits_{x\in X}|I(x)|
\geq \sum\limits_{x\in X}\frac{f(x)+2}{4}
=\frac{1}{4}\left(\alpha_b(G)+2|X|\right)
\geq \frac{1}{4}\alpha_b(G)\left(1+\frac{2}{f_{\max}}\right),
\end{eqnarray}
and, hence, 
\begin{eqnarray}\label{e4}
\alpha_b(G)\leq 4\left(1-\frac{2}{f_{\max}+2}\right)\alpha(G).
\end{eqnarray}

\bigskip

\noindent {\bf Case 2} $X_0=\emptyset$ and $X_3\not=\emptyset$.

\bigskip 

\noindent Let $x_3$ be some vertex in $X_3$.
By (B1), we may assume that $p(x_3)$ is chosen in such a way
that it has a neighbor $y_3$ outside of $B(x_3)$.
Suppose, for a contradiction,
that $y_3$ belongs to $B(x)$ for some $x$ in $X$.
If 
$f(x_3)\geq f(x)$, then
$${\rm dist}_G(x_3,x)\leq |P(x_3)|+|P(x)|-1
\leq \left\lfloor\frac{f(x_3)+2}{2}\right\rfloor+\left\lfloor\frac{f(x)+2}{2}\right\rfloor-1
\leq \frac{f(x_3)+1}{2}+\frac{f(x_3)+1}{2}-1
=f(x_3),$$
which contradicts (B2),
and, if 
$f(x_3)<f(x)$, then $X_0=\emptyset$ implies $f(x_3)\leq f(x)-2$, 
and, hence,
$${\rm dist}_G(x_3,x)\leq |P(x_3)|+|P(x)|-1
\leq \left\lfloor\frac{f(x_3)+2}{2}\right\rfloor+\left\lfloor\frac{f(x)+2}{2}\right\rfloor-1
\leq \left\lfloor\frac{f(x)}{2}\right\rfloor+\left\lfloor\frac{f(x)+2}{2}\right\rfloor-1
\leq f(x),$$
which contradicts (B2).
Hence
$$y_3\in R.$$ Let $I=\{ y_3\}\cup\bigcup\limits_{x\in X}I(x)$.
Suppose, for a contradiction, that $I$ is not independent.  In view of
the argument in Case 1, it follows that $y_3$ is adjacent to a vertex
$p(x)$ for some $x$ in $X$. As $p(x)$ has a neighbor outside of
$B(x)$, we have  $x\in X_0\cup X_1=X_1$ in this case.
If $f(x_3)\geq f(x)$, then $f(x)\leq f(x_3)-2$, and
$${\rm dist}_G(x_3,x)\leq |P(x_3)|+|P(x)|
\leq \frac{f(x_3)+1}{2}+\frac{f(x)+1}{2}
\leq \frac{f(x_3)+1}{2}+\frac{f(x_3)-1}{2}
=f(x_3),$$
which contradicts (B2), and,
if $f(x_3)\leq f(x)$, then $f(x_3)\leq f(x)-2$, and 
$${\rm dist}_G(x_3,x)\leq |P(x_3)|+|P(x)|
\leq \frac{f(x_3)+1}{2}+\frac{f(x)+1}{2}
\leq \frac{f(x)-1}{2}+\frac{f(x)+1}{2}
=f(x),$$ 
which contradicts (B2).
Hence, $I$ is independent.
Since $X_0=\emptyset$, by Table~\ref{tab:values} we obtain
$$|I(x)|\geq \frac{f(x)+1}{4}$$
for every $x$ in $X$.
As before, $f_{\max}\cdot |X|\geq \alpha_b(G)$, and, hence,
$$
\alpha(G)\geq 1+\sum\limits_{x\in X}|I(x)|
\geq 1+\sum\limits_{x\in X}\frac{f(x)+1}{4}
=1+\frac{1}{4}\left(\alpha_b(G)+|X|\right)
\geq 1+\frac{1}{4}\alpha_b(G)\left(1+\frac{1}{f_{\max}}\right),
$$
which implies
\begin{eqnarray}\label{e5}
\alpha_b(G)\leq 4\left(1-\frac{1}{f_{\max}+1}\right)(\alpha(G)-1).
\end{eqnarray}

\bigskip

\noindent {\bf Case 3} $X_0\not=\emptyset$.

\bigskip 

\noindent Let $x_0$ be some vertex in $X_0$,
and let
$$I=I(x_0)\cup \bigcup\limits_{x\in X_0\setminus \{
  x_0\}}I(x)\setminus \{ p(x)\} \bigcup\limits_{x\in \cup X_1\cup
  X_2\cup X_3}I(x).$$ Exactly as in Case 1, it follows that
$I\setminus \{ p(x_0)\}$ is independent.  Suppose, for a
contradiction, that $I$ itself is not independent.  This implies that
the vertex $p(x_0)$, which lies in $I(x_0)$, is adjacent to a vertex
$p(x)$ for some $x$ in $X$.  As $p(x)\in I$ and $p(x)$ has a neighbor
outside of $B(x)$, we have $x\in X_1$.  So if $f(x)\geq f(x_0)$, then
$f(x_0)\leq f(x)-1$ and, hence,
$${\rm dist}_G(x_0,x)\leq |P(x)|+|P(x_0)|-1 \leq \frac{f(x)+1}{2}+\frac{f(x_0)+2}{2}-1
\leq \frac{f(x)+1}{2}+\frac{f(x)+1}{2}-1=f(x),$$
which contradicts (B2),
and, 
if $f(x)\leq f(x_0)$, then $f(x)\leq f(x_0)-3$ and, hence,
$${\rm dist}_G(x_0,x)\leq |P(x)|+|P(x_0)|-1 \leq \frac{f(x)+1}{2}+\frac{f(x_0)+2}{2}-1
\leq \frac{f(x_0)-2}{2}+\frac{f(x_0)+2}{2}-1<f(x_0),$$
which again contradicts (B2).
Hence, $I$ is independent.
Since
$|I(x)\setminus \{ p(x)\}|=\frac{f(x)}{4}$ for $x$ in $X_0$, and
$|I(x)|>\frac{f(x)}{4}$ for $x$ in $X\setminus X_0$, 
we obtain
\begin{eqnarray}
\alpha(G)\geq |I|
=1
+\sum\limits_{x\in X_0}|I(x)\setminus \{ p(x)\}|
+\sum\limits_{x\in X\setminus X_0}|I(x)|
\geq 1+\sum\limits_{x\in X}\frac{f(x)}{4}
=1+\frac{\alpha_b(G)}{4},\label{e3}
\end{eqnarray}
and, hence, 
\begin{eqnarray}\label{e6}
\alpha_b(G)\leq 4\alpha(G)-4.
\end{eqnarray}
Note that the inequality (\ref{e6}) is always strictly weaker 
than the inequality (\ref{e5}), and hence, 
the three inequalities (\ref{e4}), (\ref{e5}), and (\ref{e6})
together imply (\ref{e2}).

\bigskip

\noindent We proceed to the characterization of the extremal graphs.
Lemma \ref{lemma1} implies that all graphs in 
${\cal G}_0\cup {\cal G}_2$ satisfy (\ref{e2}) with equality.
Now, let $G$ and $f$ be such that (\ref{e2}) holds with equality.
Since equality in (\ref{e2}) can not be achieved in Case 2, 
either Case 1 or Case 3 applies to $G$.

We consider two cases.

\bigskip

\noindent {\bf Case A} {\it Either $2\alpha(G)>f_{\max}+2$,
or $2\alpha(G)\leq f_{\max}+2$ and Case 3 applies to $G$.}

\bigskip 

\noindent Since 
$2\alpha(G)>f_{\max}+2$ implies
$4\left(1-\frac{2}{f_{\max}+2}\right)\alpha(G)<4\alpha(G)-4$,
necessarily Case 3 applies to $G$, 
and we use the notation from that case.
It follows that (\ref{e6}), and, hence, also (\ref{e3}) hold with equality.
Since $|I(x)|>\frac{f(x)}{4}$ for $x$ in $X\setminus X_0$, 
this implies 
$$X=X_0.$$
We may assume that $x_0$ was chosen such that $f(x_0)=f_{\max}$.

If $f(x_1)<f_{\max}$ for some $x_1$ in $X$, 
then $f(x_1)\leq f(x_0)-4$.
Suppose, for a contradiction, that $p(x_0)$ and $p(x_1)$ are adjacent.
In this case 
$${\rm dist}_G(x_0,x_1)\leq \frac{f(x_0)+4}{4}+\frac{f(x_1)+4}{4}-1
\leq \frac{f(x_0)+4}{4}+\frac{f(x_0)}{4}-1
=f(x_0),$$
which contradicts (B2). 
Hence, $I\cup \{ p(x_1)\}$ is independent, 
which implies the contradiction $\alpha(G)>|I|$.
Hence, 
$$f(x)=f_{\max} \mbox{ for every $x$ in $X$.}$$
Let the integer $k$ be such that $f_{\max}=4k$.

If there is some $x$ in $X$ such that $\partial B(x)$ contains 
two nonadjacent vertices $p$ and $p'$, 
then $(I\setminus \{ p(x_0)\})\cup \{ p,p'\}$
is independent, 
which implies the contradiction $\alpha(G)>|I|$.
Hence, $\partial B(x)$ is a clique for every $x$ in $X$.
If there are two distinct vertices $x$ and $x'$ in $X$
for which $p(x)$ and $p(x')$ are not adjacent,
then $(I\setminus \{ p(x_0)\})\cup \{ p(x),p(x')\}$
is independent, 
which implies the contradiction $\alpha(G)>|I|$.
Since $p(x)$ was an arbitrary vertex in $\partial B(x)$,
it follows that
$$\bigcup\limits_{x\in X}\partial B(x)\mbox{ is a clique.}$$
Since $G$ has diameter at least $3$ or $\alpha(G)\geq 3$,
and $f$ is an optimal broadcast on $G$,
it follows that 
$$|X|\geq 2.$$
If $R$ is not a clique, 
then adding two nonadjacent vertices from $R$
to $I\setminus \{ p(x_0)\}$ yields an independent set, 
which implies the contradiction $\alpha(G)>|I|$.
Hence,
$$\mbox{$R$ is a clique}.$$
If some vertex $p$ in $\bigcup\limits_{x\in X}\partial B(x)$
is not adjacent to some vertex $y$ in $R$,
then we may assume that 
$x_0$ and $p(x_0)$ have been chosen such that $p(x_0)=p$,
and $I\cup \{ y\}$ is independent,
which implies the contradiction $\alpha(G)>|I|$.
Hence,
$$\mbox{$R$ is completely joined to $\bigcup\limits_{x\in X}\partial B(x)$}.$$
Let $x$ be an arbitrary vertex in $X$, 
and let $H=G[B(x)\setminus \partial B(x)]$.
Recall that 
$$B(x)\setminus \partial B(x)
=B_0(x)\cup \ldots \cup B_{2k-1}(x),$$
that $B_0(x)$ contains only $x$,
and that there are no edges between $B_i(x)$ and $B_j(x)$
if $|j-i|\geq 2$.

If $\alpha(H)>k$,
then we may assume that $x_0$ is distinct from $x$,
and adding 
a maximum independent set in $H$
to the set $I\setminus (I(x)\setminus \{ p(x)\})$
yields an independent set in $G$,
which implies the contradiction $\alpha(G)>|I|$.
Hence,
$$\alpha(H)=k.$$
If $B_i(x)$ is not a clique for some $i$ in $\left[2k-1\right]$, 
then a set containing 
\begin{itemize}
\item two nonadjacent vertices from $B_i(x)$, and 
\item one vertex from $B_j(x)$ for every $j$ in $\left[2k-1\right]_0$
such that $j$ and $i$ have the same parity modulo $2$
\end{itemize}
is an independent set in $H$ with more than $k$ vertices, 
which is a contradiction.
Hence, 
$$B_i(x)\mbox{ is a clique for every $i$ in $\left[2k-1\right]_0$.}$$
If there is an even integer $i$ in $\left[2k-1\right]$
such that some vertex $x$ in $B_i(x)$ 
is not adjacent to some vertex $x'$ in $B_{i+1}(x)$,
then a set 
\begin{itemize}
\item containing $x$ and $x'$,
\item one vertex from $B_j(x)$ for every even $j$ 
in $\left[2k-1\right]_0$ less than $i$,
and 
\item one vertex from $B_j(x)$ for every odd $j$ 
in $\left[2k-1\right]_0$ larger than $i+1$
\end{itemize}
is an independent set in $H$ with more than $k$ vertices, 
which is a contradiction.
Hence,
$$\mbox{$B_{2i}(x)$ is completely joined to $B_{2i+1}(x)$
for every $i$ in $[k-1]_0$.}$$
Since $x$ was an arbitrary vertex in $X$, 
at this point it follows 
that $G$ contains a graph $G_0$ from ${\cal G}_0(k,\ell)$ 
with $\ell=|X|$ as a spanning subgraph.
Since
adding any further edge $e$ to $G_0$ 
such that $G_0+e\not\in {\cal G}_0(k,\ell)$
results in a graph 
that has less than $\ell$ vertices of eccentricity $f_{\max}=4k$,
we obtain $G\in {\cal G}_0(k,\ell)$,
which completes the proof
in this case.

\bigskip

\noindent {\bf Case B} {\it $2\alpha(G)\leq f_{\max}+2$ and Case 1 applies to $G$.}

\bigskip 

\noindent We use the notation from Case 1.
Since $4\left(1-\frac{2}{f_{\max}+2}\right)\alpha(G)\geq 4\alpha(G)-4$,
it follows that (\ref{e4}), and, hence, also (\ref{e7}) hold with equality.
This implies $f_{\max}\cdot |X|=\alpha_b(G)$, and, hence,
$$f(x)=f_{\max} \mbox{ for every $x$ in $X$.}$$
Furthermore,
since $|I(x)|>\frac{f(x)+2}{4}$ for $x$ in $X_1$, 
equality in (\ref{e7}) implies 
$$X=X_2.$$
Let the integer $k$ be such that $f_{\max}=4k+2$.

As in Case A, we have $|X|\geq 2$.
If $|X|\geq 3$, then, by (\ref{e7}), 
$\alpha(G)\geq 3(\frac{f_{\max}+2}{4})\geq 3(k+1)$, and, hence,
$2\alpha(G)\geq 6k+6>4k+4=f_{\max}+2$.
Hence,
$$|X|=2.$$
If $R$ is not empty, then adding a vertex from $R$ to $I$
yields an independent set, 
which implies the contradiction $\alpha(G)>|I|$.
Hence,
$$\mbox{$R$ is empty}.$$
Let $X=\{ x_1,x_2\}$,
and let $B_i^j=B_i(x_j)$ for every $i$ in $[2k+1]_0$ and $j$ in $[2]$,
cf. the definition of the graphs in ${\cal G}_2(k)$.
Arguing similarly as in Case A, we obtain
that
$$\mbox{$B^j_i$ is a clique for every $i$ in $[2k+1]_0$ and $j$ in $[2]$,}$$
and that 
$$\mbox{$B^j_{2i}$ is completely joined to $B^j_{2i+1}$
for every $i$ in $[k]_0$ and $j$ in $[2]$.}$$
Since $G$ is connected,
$$\mbox{there are some edges between $B_{2k+1}^1$ and $B_{2k+1}^2$.}$$
Again, it follows that $G$ contains a graph $G_2$ from ${\cal G}_2(k)$
as a spanning subgraph.  Since adding any further edge $e$ to $G_2$
such that $G_2+e\not\in {\cal G}_2(k)$ results in a graph of diameter
less than $4k+3$, we obtain $G\in {\cal G}_2(k)$, which completes the
proof.
\end{proof}

\end{document}